\documentclass[11pt]{article}       
\usepackage{geometry}               
\usepackage{graphicx}
\usepackage{caption}
\usepackage{subcaption}
\usepackage{comment}
\usepackage{amsmath} 
\usepackage{amssymb}  
\usepackage{amsthm}  
\usepackage{bm}
\usepackage{lipsum}
\usepackage{color}
\usepackage{multirow}
\usepackage{array}
\usepackage{cite}
\usepackage{hyperref}

\newtheorem{prop}{Proposition}
\newtheorem{defn}{Definition}
\newtheorem{cor}{Corollary}
\newtheorem{lem}{Lemma}
\newtheorem{thm}{Theorem}

\newtheorem{assumption}{Assumption}

\numberwithin{equation}{section}

\newcommand{\R}{\mathbb{R}}

\newcommand{\de}{\mathrm{d}}
\newcommand{\x}{\bm{x}}
\newcommand{\Conv}{\mathrm{Conv}}

\DeclareMathOperator*{\argmin}{arg\,min}

\title{On Representation Formulas for  Optimal Control: \\
A Lagrangian Perspective\thanks{This work was supported in part by the BK21 FOUR program of the Education and Research Program for Future ICT Pioneers, Seoul National University in 2022,
the Information and Communications Technology Planning and Evaluation (IITP) grant funded by MSIT(2020-0-00857),   the National Research Foundation of Korea funded by MSIT(2020R1C1C1009766), and Samsung Electronics.}}

\author{
Yeoneung Kim\thanks{Department of Electrical and Computer Engineering, ASRI,  Education and Research program for Future ICT Pioneers, 
 Seoul National University, Seoul 08826, South Korea.}
 \thanks{Deparment of Financial Mathematics, Gachon University, Seongnam 13120, South Korea.}
 \and
 Insoon Yang\footnotemark[2]}

\date{}

\begin{document}
\maketitle

\pagestyle{myheadings}
\thispagestyle{plain}

\begin{abstract}
In this paper, we study representation formulas for finite-horizon optimal control problems with or without state constraints, unifying two different viewpoints: the Lagrangian and dynamic programming (DP) frameworks. In a recent work \cite{lee2021computationally}, the generalized Lax formula is obtained via DP  for optimal control problems with state constraints and nonlinear systems. 
We revisit the formula from the Lagrangian perspective to provide a unified framework for understanding and implementing the nontrivial representation of the value function. 
Our simple derivation  makes  direct use of the Lagrangian formula from the theory of Hamilton--Jacobi (HJ) equations. 
We also discuss a rigorous way to construct an optimal control using a $\delta$-net,   as well as a numerical scheme for controller synthesis via convex optimization. 
\end{abstract}


\section{Introduction}

A popular approach to solving and analyzing continuous-time optimal control problems is to use the Hamilton--Jacobi--Bellman (HJB) equation, which is obtained via dynamic programming (DP)~\cite{bellman1966dynamic}.
The theory of HJB equations has been extensively studied, particularly in the viscosity solution framework~\cite{crandall1992user}.
The value function of an optimal control problem corresponds to the unique viscosity solution of an associated HJB equation. 
Various methods have been proposed to numerically solve HJB equations by discretizing the state space~(e.g., \cite{osher1991high, jiang2000weighted, barles2002convergence,forsyth2007numerical}).
However, in general, convergent numerical methods have a scalability issue since the computational complexity increases exponentially with the dimension of the state space.

A notable feature of the viscosity solution  is that it can be expressed using a variational formula if the Hamiltonian is convex.
In particular, when the Hamiltonian depends only on the costate, we have a simple representation of the solution, called the Hopf--Lax formula \cite{hopf1965generalized,bardi1984hopf}. 
More recently, various algorithms using the Hopf--Lax type representation formulas  have been proposed and shown to efficiently solve some classes of high-dimensional HJB equations~\cite{darbon2016algorithms, aliyu2016modified,chow2019algorithm, darbon2020decomposition, lee2020hopf, yegorov2021perspectives}. 
Furthermore, there have been a few attempts to identify explicit solutions to a certain class of optimal control problems using representation formulas related to the Hopf--Lax formula \cite{chen2021lax}. 
Another notable work is \cite{lee2021computationally}, where a generalized Lax formula is obtained via DP to handle both state constraints and nonlinear systems. 
Unfortunately, the control trajectory obtained by this method presents chattering behaviors.

In this paper, we revisit the generalized Lax formula from the Lagrangian perspective to provide a unified view.
We first show that the generalized Lax formula can be directly derived using the Lagrangian framework from the theory of HJ equations. 
For optimal control problems without state constraints, the generalized Lax formula corresponds to the Lagrangian formula under a simple change of variables. 
 To extend to state-constrained problems, we introduce a penalty function that penalizes the Lagrangian outside of a prescribed domain, as in \cite{capuzzo1990hamilton}, which in turn relates viscosity solutions to state-constrained viscosity solutions \cite{soner1986optimal}. 
 Under some structural assumptions, we also discuss a concrete way to construct an optimal control  in the form of a simple function using the notion of $\delta$-net with the optimality gap bounded by $C \delta$ for some positive constant $C$. 
 Another important observation is that the Lagrangian framework provides a natural representation of the value function using the convex conjugate of the Hamiltonian, even if the original optimal control problem is nonconvex. 
 Exploiting this property, we propose a controller synthesis scheme via convex optimization. 
 The results of our numerical experiments show that the proposed scheme significantly reduces oscillations in control trajectories while achieving slightly lower total costs compared to the method proposed in \cite{lee2021computationally}.

This paper is organized as follows. In Section 2, we revisit the generalized Lax formula using the Lagrangian framework from the theory of HJ equations. In Section 3, a theoretically rigorous way to construct   optimal controls is discussed, together with its algorithmic implications. Section 4 presents a controller synthesis scheme and its application to a state-constrained optimal control example.

\section{Optimal Control and the Generalized Lax Formula}

In this section, an equivalent representation of the value function for an optimal control problem is provided from the Lagrangian perspective in the theory of HJ equations. After the connection is established between the generalized Lax formula and the HJ PDE from the Lagrangian viewpoint, we extend the relationship to state-constrained optimal control problems.

Consider a continuous-time dynamical system of the form\footnote{Throughout the paper, $\dot{x}(t)$ denotes $\frac{\de}{\de t}x(t)$.}
\begin{equation}\label{sys}
\begin{split}
\dot{x} (t) &= f(t, x(t), u(t)), \quad 0 \leq t \leq T,
\end{split}
\end{equation}
where $x(t) \in \R^n$ and $u(t) \in U \subset \R^m$ are the system state and control input at time $t$, respectively, 
The set of admissible controls is defined as
\[
\mathcal{U}:=\{u:[0,T] \rightarrow U  :  \text{$u$ is measurable}\},
\]
where the control set $U$ is assumed to be compact.

A finite-horizon optimal control problem can then be formulated as 
\[
\min_{u \in \mathcal{U}} \left \{ \int_0^T L(t,x(t),u(t)) \de t+g(x(T)) : x(0) = \bm{x} \right \},
\]
where $L$ is the stage-wise cost function, and $g$ is the terminal cost function. 
Throughout the paper, we always assume that
\begin{itemize}
    \item $f(t, \x, a)$ and $L(t, \x, a) \geq 0$ are Lipschitz continuous in $t$ and $\x$ for each $a \in \mathcal{U}$,
    \item $g(\x)$ is Lipschitz continuous and uniformly bounded.
\end{itemize}
Let $V(\bm{x},t)$ denote the optimal value function, defined as
\begin{equation}\label{value}
V(\bm{x},t):=\inf_{u \in \mathcal{U}} \left \{  \int_t^T L(s,x(s),u(s)) \de s+g(x(T)) : x(t) = \bm{x} \right \},
\end{equation}
which represents the optimal cost-to-go starting from $x(t) = \bm{x}$ at time $t$. 
Under the assumptions above, it is well known that $V$ is the unique viscosity solution of the following HJ equation~\cite{bardi1997optimal}:\footnote{Throughout the paper, $V_t$ denotes $\frac{\partial}{\partial t}V(\bm{x}, t)$ and $DV$ denotes $\frac{\partial}{\partial \bm{x}} V(\bm{x}, t)$.}
\begin{align*}
\begin{cases}
    -V_t + H(t, \bm{x},DV) = 0 &\quad \text{in} \quad \R^n\times (0,T),\\
    V(\bm{x},T)=g(\bm{x}) &\quad  \text{on} \quad \R^n,
    \end{cases}
\end{align*}
where the Hamiltonian is given by
\[
H(t, \bm{x}, p):=\sup_{a \in U} \left\{ -p\cdot f(t,\bm{x}, a)-L(t, \bm{x}, a) \right\}.
\] 

The following feasible set of negative vector fields plays an important role in the change of variables with respect to the control action.
\begin{defn}
Given $\bm{x} \in \mathbb{R}^n$ and $t \in [0,T]$, we define 
\[
B(\bm{x}, t):=\{-f(t, \bm{x}, a): a \in U\},
\]
 which we call the feasible set of negative vector fields.
\end{defn}

Consider the following change of variables: 
\[
\beta (t) := -f(t, x(t), u(t)) \in B(x(t), t).
\]
Then, the dynamical system~\eqref{sys} and the value function~\eqref{value} can be rewritten as
\[
\dot{x} (t) = -\beta (t), \quad 0 \leq t \leq T,
\]
and
\[
V(\bm{x}, t) = \inf_{\beta \in \mathcal{B}} \left \{  \int_t^T \tilde{L}(s,x(s),\beta(s)) \de s+g(x(T)) : 
x(t) = \bm{x},
\beta(s) \in B(x(s), s)
 \right \},
\]
where
\[
\tilde L(s, \bm{x},b):=\min_{a\in U} \{L(s,\bm{x},a) : f(s,\bm{x},a)=-b\},
\]
and 
$\mathcal{B}:=\{\beta:[0,T] \rightarrow \mathbb{R}^n  :  \text{$\beta$ is measurable}\}$.

To consider the formulation above from the Lagrangian perspective, we will use the convex conjugate $H^*(t, \x, b):=\sup_{p \in \R^n} \left \{ p\cdot b - H(t, \x, p) \right \}$ of the Hamiltonian as the Lagrangian. 
We first identify the domain of $b$ in which $H^*(t, \x, b)$ is well defined. 
Note that 
\[
H(t, \bm{x}, p)= \sup_{b\in B(\bm{x}, t)} \left \{ p\cdot b - \tilde L(t, \bm{x},b) \right \}.
\]
It is clear that $H(t, \bm{x},\cdot)$ is convex as it is the supremum of affine functions in $p$.  
Since $B(\x, t)$ is compact, 
\begin{equation*}
    \sup_{b\in B(\x, t)} p\cdot b- C \leq H(t,\x,p) \leq \sup_{b\in B(\x,t)} p\cdot b+C
\end{equation*}
for some positive constant $C$ depending only on $\bm{x}$ and $t$. It follows from the linearity of $p\cdot b$ with respect to $b$  that
\begin{equation*}
    \sup_{b\in B(\x,t)} p\cdot b = \sup_{b\in \Conv(B(\x,t))} p\cdot b, 
\end{equation*}
where $\Conv (A)$ denotes the convex hull of a set $A$. 
Based on the assertion above, we identify a necessary and sufficient condition in which $H^*$ is well defined. 
\begin{prop}\label{prop:well}
Given $\x \in \mathbb{R}^n$ and $t \in [0,T]$, 
let 
\begin{equation}\label{conj}
H^*(t, \x, b):=\sup_{p \in \R^n} \left \{ p\cdot b - H(t, \x, p) \right \}
\end{equation}
be the convex conjugate of $H(t, \x, \cdot)$. 
 Then, $H^*(t, \x, b)= +\infty$ for $b\not\in \Conv(B(\x,t))$ and $H^*(t, \x, b) <\infty$ for $b\in \Conv(B(\x,t))$.
\end{prop}
\begin{proof}
Suppose first that $b\not\in \Conv(B(\x,t))$.
Recall that $H(t,\x,p) \leq \sup_{b\in \Conv(B(\x,t))} p\cdot b +C$. Therefore,  we have
\[
H^*(t, \x,b) \geq \sup_{p} \left \{ p\cdot b - \sup_{b' \in \Conv(B(\x,t))} p\cdot b' \right \}-C.
\]
Consider a hyperplane passing through $b$ such that $B( \x, t)$ lies strictly above (or below) of the plane. Denoting the unit normal vector pointing toward $B(\x, t)$ by $e$, we have that
\[
e \cdot(b-b') \geq \delta >0
\]
for all $b' \in B(\x, t)$ and some $\delta>0$. Taking $p=z e$ and letting $z$ tend to $\infty$, we obtain that $H^*(t, \x, b)= +\infty$ when $b\not\in \Conv(B(\x,t))$.

The second part is a straightforward result following from the definition of the Legendre transform. However, we present a proof for the sake of completeness. 
We first write $b=\sum_{i=1}^N \gamma_i b_i \in \Conv(B(\x,t))$ with $b_i \in B(\x, t)$ and $\sum_{i=1}^N \gamma_i =1$, where $\gamma_i$'s are nonnegative.
Then, we deduce that
\begin{align}\label{finite}
    H^*(t,\x,b) & = \sup_p  \left \{ p\cdot b - H(t,\x,p) \right \}\nonumber\\
&\leq \sum_{i=1}^N \gamma_i \sup_p \left \{ p\cdot b_i - H(t, \x,p) \right \}.
\end{align}
Since $H(t, \x, p)=\sup_{b'\in B(\x,t)} \{ p\cdot b' - \tilde L(s,\x,b') \}$, it follows that
\begin{equation*}
    H(t,\x,p) \geq p\cdot b' - \tilde L(t, \x,b')
\end{equation*}
for any $b' \in B(\x,t)$. Choosing $b'=b_i$ for each $i$ and using \eqref{finite}, we conclude that $H^*(t,\x,b) \leq  \sum_{i=1}^N  \gamma_i \tilde L (t, \x, b_i)$, where the right-hand side is finite.
\end{proof}

This proposition implies that the condition $b\in \Conv(B(\x,t))$ is necessary and sufficient for $H^*$ to be well-defined, or finite. 
Thus, if  $H^* (t,\x, b)$ is used as the stage-wise cost function, this condition forces the negative vector field to stay inside of $B(x(t),t)$ so that $\dot x(t)=f(t,x(t),u(t))$. This motivates us to use the Lagrangian formula using $H^*$ as the stage-wise cost function.

\subsection{Representation Formulas for Unconstrained Optimal Control}
Before stating our main observation, let us recall the Lagrangian framework \cite{tran2021hamilton}. It is known that for a Lagrangian $\mathcal{L}(t, \x, b)$ which is convex in $b$, the value function
\begin{equation} \label{lag}
    \phi(x,t):=\inf \left \{\int_0^t \mathcal{L}(s,\xi(s),\dot\xi(s))\de s +g(\xi(0)) : \xi(t)=\bm{x},\dot \xi(s) \in L^1([0,t]) \right \}
\end{equation}
is the  unique viscosity solution of the following HJ PDE:
\begin{align*}
\begin{cases}
\phi_t+H(t, \x,D \phi)=0 &\quad \text{in}\quad  \R^n \times (0,\infty),\\
\phi (\x,0)=g( \x) &\quad  \text{on} \quad  \R^n,
\end{cases}
\end{align*}
where $H(t,\x,p):=\sup_b \{ p\cdot b - \mathcal{L}(t,\x,b) \}$.

Now regarding $H^*$ as the Lagrangian, we obtain the following representation of the value function~\eqref{value}.

\begin{thm}\label{unconstraint}
For $(\bm{x},t) \in \mathbb{R}^n \times [0,T]$,
the value function~\eqref{value} can be represented as
\begin{align}
    V(\x,t)&=\inf \left \{\int_t^T H^*(s,x(s),-\dot x(s)) \de s +g(x(T)): x(t) = \x \right \}\label{lagrange}\\
    &=\inf \left \{\int_t^T H^{*}(s,x(s),\beta(s)) \de s+g(x(T)): x(t) = \x, \beta(s)\in \Conv(B(x(s),s)) \right\}.\label{dual}
\end{align}
\end{thm}
\begin{proof}
Let us define
\begin{equation*}
    W(\x,t):=\inf \left \{ \int_t^T H^*(s,x(s),-\dot x(s))ds +g(x(T))): x(t) = \x \right \}.
\end{equation*}
It follows form  the Lagrangian formula~\eqref{lag} that $W$ solves  the following HJ equation in the viscosity sense:
\begin{align*}
\begin{cases}
    -W_t + H^{**}(t,\x,DW)=0 &\quad \text{in}\quad  \R^n \times (0,T),\\
    W(\x,T)=g(\x) &\quad \text{on} \quad \R^n,
    \end{cases}
\end{align*}
where $H^{**}(t, \x, p):=\sup_{b \in \R^n} \left \{ b\cdot p - H^*(t, \x, b) \right \}$ is the convex conjugate of $H^*(t, \x, \cdot)$.

By the definition of $H^*$, we first notice that for any $p, b \in \mathbb{R}^n$
\[
H^* (t, \x, b) \geq p\cdot b - H (t, \x, p),
\]
which implies that 
\[
H (t, \x, p) \geq p\cdot b - H^* (t, \x, b).
\]
Taking supremum of both sides with respect to $b$ yields 
$H \geq H^{**}$. 

To show $H \leq H^{**}$, 
we first observe that
\begin{align*}
    H^{**}(t, \x,p)&=\sup_{b\in \R^n} \left \{ p\cdot b - H^*(t, \x,b)\right \}
    \\&=\sup_{b\in \R^n} \left [ p \cdot b - \sup_{p'} \left \{ p' \cdot b - H(t,\x,p') \right \} \right ]\\
    &=
    \sup_{b\in \R^n} \inf_{p'} \left \{ (p-p')\cdot b + H(t, \x,p') \right \}.
\end{align*}
Therefore, 
\begin{equation*}
    H^{**}(t, \x,p) \geq \inf_{p'}  \left \{ H(t,\x, p' )-(p'-p)\cdot b\right \}
\end{equation*}
for all $b \in  \R^n.$ 
 Choosing $b= D_p H(t,\x,p)$, it follows from the convexity of $H(t, \x, \cdot)$  that
\begin{align*}
    H(t, \x, p')-(p-p')\cdot b &= H(t,\x,p') - (p'-p)\cdot D_p H(t,\x,p)
    \\& \geq H(t,\x,p),
\end{align*}
which implies that $H^{**} \geq H$. 
Therefore, we have  $H^{**} =H$. By the uniqueness of the viscosity solution, we conclude that $V\equiv W$.

To show that the second formula~\eqref{dual} is valid, 
we recall Proposition~\ref{prop:well} implying that
$H^*(t, \x,b) < \infty$ if and only if $b\in \Conv(B(\x,t))$. 
The second formula directly follows from the definition of  $B(x(s), s)$ and Proposition~\ref{prop:well}.
\end{proof}

Our representation formula is equivalent to the one in \cite{lee2021computationally}. 
However, departing from their DP approach, 
our simple derivation makes direct use of the Lagrangian formula, taking the convex conjugate of the Hamiltonian as the Lagrangian. 
Another important observation is that the Lagrangian framework provides a natural representation of the value function with a convex function $H^*$ even if the original stage-cost function $L$  is nonconvex. 
Exploiting this property, we propose a controller synthesis scheme via convex optimization in Section~\ref{sec:num}.
 
Now, our main question is whether the optimal solution of \eqref{dual} yields to obtain a feasible control $u(t) \in U$. In general, the convex conjugate of the Hamiltonian can be represented as
\begin{equation*}
    H^*(t,\x,b)=\sum_{i=1}^N \gamma_i \tilde L (t,\x,b_i)
\end{equation*}
with $b=\sum_{i=1}^N \gamma_i b_i \in \Conv(B(\x,t))$, where $b_i \in B(\x,t)$, $\gamma_i \geq 0$ and $\sum_{i=1}^N \gamma_i=1$. This representation comes from the duality property of the Legendre transform~\cite{rockafellar2009variational}. However, it is unclear whether an optimal $\beta(t) \in \Conv(B(x(t),t))$ induces an admissible control $u(t) \in U$. We tackle this issue in Section~\ref{sec:control}.

\subsection{Representation Formulas for Constrained Optimal Control} 

We now extend the representation formula~\eqref{dual} to the case of state-constrained optimal control. 
Specifically, the  state trajectory of \eqref{sys} must lie in a convex compact set $\bar\Omega$, namely, $x(t) \in \bar \Omega$ for any $t$ and for some bounded open set $\Omega \subset \R^n$. 
The value function for the state-constrained optimal control problem is defined as
\begin{equation}\label{value_c}
V(\bm{x},t):=\inf_{u \in \mathcal{U}} \left \{  \int_t^T L(s,x(s),u(s)) \de s+g(x(T)) : x(t) = \bm{x}, x(s) \in \bar{\Omega} \right \}.
\end{equation}
We further assume that the state trajectory can be driven into $\bar \Omega$ with some control $u \in \mathcal{U}$ at the boundary of $\Omega$, which can be stated as follows:
\begin{assumption}
There exists a positive constant $\nu$ such that for all $\x \in \partial \Omega$ and  $t\in[0,T]$, there exists $a \in U$ satisfying $f(t,\x,a) \cdot n(\x) \leq -\nu$, where the vector $n(\x) \in \mathbb{R}^n$ is normal to $\partial \Omega$  at $\x$.
\end{assumption}
Under this assumption, by the classical result on state-constrained HJ equation \cite{capuzzo1990hamilton}, $V$ is the unique viscosity solution in $\Omega \times (0,T)$ and the viscosity supersolution on $\bar \Omega \times (0,T)$ of 
\begin{align*}
\begin{cases}
    -V_t +H(t, \x,DV) = 0 &\quad \text{in} \quad \Omega \times (0,T),\\
    V(\x,T)=0 &\quad \text{on} \quad \bar \Omega,
    \end{cases}
\end{align*}
where the Hamiltonian $H$ is the same as before. Such $V$ is called the state-constraint viscosity solution, proposed by \cite{soner1986optimal}. 
Taking the convex conjugate $H^*$ as the Lagrangian, we can derive 
 representation formulas for the state-constrained value function $V$,
 similar to \eqref{lagrange} and \eqref{dual}.

\begin{thm}\label{constrained}
For $(\bm{x},t) \in \bar \Omega \times [0,T]$, 
the state-constrained value function $V$~\eqref{value_c} can be represented as
\begin{align}
    V(\x,t)&=\inf \left \{\int_t^T H^*(s,x(s),-\dot x(s)) \de s +g(x(T)): x(t) = \x, x(s) \in \bar\Omega \right \}\label{lagrange_c}\\
    &=\inf \left \{\int_t^T H^{*}(s,x(s),\beta(s)) \de s+g(x(T)): x(t) = \x, x(s) \in \bar\Omega, \beta(s)\in \Conv(B(x(s),s)) \right\}.\label{dual_c}
\end{align}
\end{thm}
\begin{proof}
Define a penalty function $\varphi( \x)$ as
\begin{equation*}
\varphi ( \x):=
\begin{cases}
0 &\text{ if } \x\in \Omega,\\
\mathrm{dist}(\x,\partial \Omega) &{\text{ if } \x\not\in\Omega},
\end{cases}
\end{equation*}
where $\mathrm{dist}(\x, A)$ denotes the Euclidean distance between a vector $\x$ and a set $A$. 
Given any $\epsilon>0$ it is known from \cite{capuzzo1990hamilton} that 
\begin{equation*}
    W_\epsilon(\x,t):=\inf_u \left [  \int_t^T \left \{H^*(s,x(s),-\dot x(s))+\frac{1}{\epsilon}\varphi (x(s)) \right \}ds+g(x(T)) : x(t) = \x \right ]
\end{equation*}
converges uniformly to $W$ as $\epsilon$ goes to $0$, where $W(x,t)$ solves
\begin{align*}
\begin{cases}
    -W_t+H(t,x,DW)=0 &\quad\text{in}\quad \Omega\times (0,T),\\
    W(x,T)=0 &\quad\text{on}\quad\bar \Omega,
\end{cases}
\end{align*}
in the viscosity sense. 
By the uniqueness of the viscosity solution, $W_\epsilon$ converges uniformly to the value function $V$ for the state-constrained optimal control problem as $\epsilon$ tends to zero.
Since $W_\epsilon(\bm{x}, t)$ converges to \eqref{lagrange_c} given $(\bm{x},t) \in \bar \Omega \times [0,T]$,
 the result follows. 
\end{proof} 

This is an interesting result that bridges the gap between the DP-based generalized Lax formula and the Lagrangian formula 
 for state-constraint optimal control problems. 
 Specifically, our result is novel in the sense that  the generalized Lax formula can be understood through the lens of state-constrained viscosity solutions.
  In  \cite{lee2021computationally}, state-constraint optimal control problems are considered with time-varying constraint sets.  Therefore, the standard state-constrained viscosity solution approach is unavailable. 
However, we consider a time-invariant constraint set $\bar \Omega$ to make use of the theory of viscosity solutions and obtain a representation formula similar to that for unconstrained optimal control problems. 
More precisely, the value function of a state-constrained optimal control problem is approximated as that of the corresponding unconstrained problem by penalizing the region outside the constraint set. 
We then use Theorem \ref{unconstraint}  and take a limit with respect to the penalty parameter. 
Such a limit exists due to the classical theory of state-constrained viscosity solutions and we arrive at the same conclusion as that of Theorem~\ref{constrained}.

\section{Construction of Optimal Controls} \label{sec:control}

In this section, we present a method for extracting an optimal control $u$ from an optimal $\beta$ obtained in evaluating the representation formula~\eqref{dual} or \eqref{dual_c}.
Our method constructs an admissible control in the form of a simple function using the notion of $\delta$-net. 
One of our main observations is that under some structural assumptions, the optimal value function  is approximated accurately as long as the controlled state trajectory is close to an optimal trajectory. 

We introduce structural assumptions needed for our analysis. They are imposed throughout this section.
\begin{assumption}\label{ass2}
The stage-cost function is separable, i.e., 
$$L(t, \x, a)=S(t, \x)+R(t, a),$$
and for each $t$, $S(t,\x)$ and $R(t,a)$ are convex in $\x$ and $a$, respectively.
\end{assumption}
Since $\x \mapsto L(t, \x, a)$ is Lipschitz continuous, so is $\x \mapsto S(t, \x)$.
The aforementioned assumption will allow us to track duality properties effectively. 
\begin{assumption}\label{ass3}
The vector field and the terminal cost function satisfy the following properties:
\begin{itemize}
    \item $f(t, \x, a)$ is affine in $\x$, that is, $$f(t, \x,a)=Ax+ h(t, a)$$ for some matrix $A \in \mathbb{R}^{n\times n}$. 
Moreover, for each $t$, $h (t, a)$ is Lipschitz continuous in $a$;
    \item $g$ is convex.
\end{itemize}
\end{assumption}

Under these assumptions,  $H^*$ can be computed explicitly  as illustrated in \cite{lee2021computationally}.
 Let 
 \[
 \tilde R(t, b):=\inf_{a\in U} \{R(t,a): -f(t,\x,a) = b\}.
 \]
Then, $\tilde L(t, \x, b)=S(t, \x)+\tilde R(t,b)$. 
Therefore, the Hamiltonian has the following separable structure:
\begin{align*}
    H(t, \x,p)&=\sup_{b\in B(\x,t)} \{p \cdot b - \tilde L (t, \x,b)\}\\
    &=-S(t,\x)+\sup_{b\in B(\x,t)} \{p \cdot b - \tilde R (t,b)\}.
\end{align*}
The convex conjugate of the Hamiltonian can be expressed as
\begin{align*}
    H^*(t, \x,b)&=\sup_p\{b\cdot p - H(t, \x,p)\}\\
    &=S(t, \x)+\sup_p \left [ b\cdot p - \sup_{b'\in B(\x,t)} \{p\cdot b' -\tilde R (t,b')\} \right ]. 
\end{align*}
Letting $H_R(t,\x,p):=\sup_{b\in B(\x,t)} \{p\cdot b - \tilde R(t,b)\}$, we have
\begin{equation*}
    H^*(t, \x,b)=S(t,\x)+ H_R^*(t,\x,b).
\end{equation*}
Now using $B(\x,t)=\{b : b=-A\x+ h(t,a), a \in U\}$, we further have
\begin{equation}\label{c_ham_exp}
    H^*(t, \x, b)=S(t, \x)+ H_f^*(t,b+A\x)
\end{equation}
where $H_f^*(t,b):= \sup_p \{p \cdot b - H_h (t,p)\}$ and $H_h(t,p):=\sup_{a \in U} \{-p \cdot h(t,a)-R(t,a)\}$.

In the following lemma, we state the relaxation theorem \cite{papageorgiou1987relaxation} that guarantees the existence of an admissible control such that the controlled trajectory can accurately approximate any state trajectory $-\dot x(s) \in \Conv(B(x(s),s))$ with same initial condition.
 \begin{lem}\label{lem:relax}
Let $x:[t,T]\rightarrow \R^n$ satisfy $- \dot x(s) \in \Conv(B(x(s),s))$. Then, for any $\epsilon>0$, there exists an admissible control $\tilde u  \in \mathcal{U}$ and a state trajectory $\tilde x : [t,T] \rightarrow \R^n$  satisfying $\dot{\tilde{x}} = f(s, \tilde{x}(s),\tilde u(s))$ for $s\in[t,T]$ and $\tilde{x}(t) = x(t)$ such that 
\[
|\tilde x(s) - x(s)| <\epsilon, \quad s \in [t, T].
\]
\end{lem}

We can then show that 
the value function $V(\x, t)$ can be approximated with an arbitrary precision using an optimal solution of \eqref{dual} or \eqref{dual_c}
although an optimal $\beta$ lies in $\Conv (B(x(s), s))$ rather than $B(x(s), s)$. 

\begin{thm}\label{main_thm}
Suppose $H_f^*(t,\cdot)$ is linear and let $(x(\cdot),\beta(\cdot))$ be an optimal solution of \eqref{dual}  (or \eqref{dual_c}). Then, for any $\epsilon>0$, there exists an admissible control $\tilde u \in \mathcal{U}$ and
a state trajectory $\tilde x : [t,T] \rightarrow \R^n$  satisfying $\dot{\tilde{x}} = f(s, \tilde{x}(s),\tilde{u}(s))$ for $s\in[t,T]$ and $\tilde{x}(t) = \x$ 
such that
\begin{equation*}
    \bigg|  \underbrace{\bigg \{\int_t^T H^*(s,x(s),\beta(s))\de s+g(x(T)) \bigg \}}_{= V(\x, t)}- \bigg \{ \underbrace{\int_t^T H^*(s,\tilde x(s),\tilde \beta(s))\de s-g(\tilde x(T))}_{\mathrm{approximation}} \bigg \}\bigg|<\epsilon,
\end{equation*}
where  $\tilde \beta(s)=-f(s,\tilde x(s),\tilde u(s))$.
\end{thm}
\begin{proof}
By Lemma~\ref{lem:relax}, for any $\varepsilon > 0$, we can select $\tilde{u} \in \mathcal{U}$ such that $|x(s)-\tilde x(s)| < \varepsilon$ for $s \in [t, T]$. 
Since $g$ is Lipschitz continuous, $|g(x(T)) - g(\tilde{x}(T))| \leq C \varepsilon$ for some positive constant $C$. 

It follows from the expression \eqref{c_ham_exp} of $H^*$ that
\begin{align*}
&H^*(s,x(s),\beta(s))-H^*(s,\tilde x(s),\tilde \beta(s))\\& = S(s,x(s))-S(s,\tilde x(s)) +H_f^*(s,\beta(s)+Ax(s)) - H_f ^* (s,\tilde \beta(s)+A \tilde x(s))\\
& \leq C \varepsilon + A ( \beta(s) - \tilde{\beta}(s))
\end{align*}
for some positive constant $C$ and some matrix $A \in \mathbb{R}^{n \times n}$
 since $\x \mapsto S(s, \x)$ is Lipschitz continuous and
$b \mapsto H_f^*(s, \x,b)$ is linear.

We now observe that
\begin{equation}\nonumber
\begin{split}
\left |\int_t^T\beta(s)-\tilde \beta(s) \de s \right |
&= \left |\int_t^T-\dot{x}(s) +  \dot{\tilde{x}}(s) \de s \right |\\
&=| x(t) - x(T) -\tilde{x} (t) + \tilde{x}(T) | \\
&\leq C \varepsilon
\end{split}
\end{equation}
for some positive constant $C$. 
Therefore, the result follows.
\end{proof}
The linearity assumption on $H_f^*(s,\x, \cdot)$ might {seem} restrictive. However,  this assumption does not require the linearity of $f$.
We discuss a class of nonlinear problems that satisfy the assumption in Section~\ref{sec:lin}.
Before doing so, we present a concrete way to recover an admissible control using a $\delta$-net in the following subsection.

\subsection{From the Representation Formulas to Approximate Optimal Controls}

We now assume that there exists an optimal solution $(x(s),\beta(s))$ of \eqref{dual} (or \eqref{dual_c}),  where $\beta(s) \in \Conv(B(x(s),s))$. 
Our goal is to show that there exists a finite set $U_\delta \subset U$  such that one can construct an admissible control $\bar{u}$ satisfying $\bar{u}(s) \in U_{\delta}$ and $|\bar x(s)- x(s)|<\epsilon$, 
where $\bar x(s)$ is the state trajectory controlled by $\bar{u}$ starting from $\bar x(t)=\x = x(t)$. If such a finite set exists, the  control set $U$ can be degenerated, which is beneficial for numerical implementation. To this end, we introduce the $\delta$-net of a given compact set $U$, denoted by $U_\delta$ for $\delta>0$.
\begin{defn}
A finite set $U_\delta$  is said to be a \emph{$\delta$-net} provided that $(i)$ for any $a,a'\in U_\delta \subseteq U$, $|a-a'| > \delta$, and $(ii)$ $\{B_\delta(a): a \in U_\delta \}$ forms a covering, i.e., $U \subset \bigcup_{a \in U_\delta} B_\delta(a)$, where $B_\delta (a)$ denotes the open ball centered at $a$ with radius $\delta$. 
\end{defn}
Since $U$ is compact, a $\delta$-net exists \cite{pollard1990empirical}.

Suppose that we choose an admissible control $u  \in \mathcal{U}$ from Lemma \ref{lem:relax}. In Proposition~\ref{prop:u}, we prove that  for any $\delta>0$ and $U_\delta$, one can find an admissible control $\bar{u}$ 
such that  $\bar u(s) \in U_\delta$  and $|\bar u(s) - u(s)| < \delta$ for $s \in [t, T]$. 
Even if we choose such $\bar u$, the measurability of this function is unclear. The following proposition yields that if $\bar u$ is constructed in an appropriate way, it is measurable on $[t,T]$. Our idea is similar to constructing a simple function that approximates a given measurable function in measure theory.
\begin{prop}\label{prop:u}
Let $(x, \beta)$ be an optimal solution of \eqref{dual} or \eqref{dual_c}. 
For any  $\delta >0$, there exist a measurable control $\bar{u} \in \mathcal{U}$ such that $\bar u(s) \in U_\delta$ and the corresponding state trajectory $\bar x: [t,T] \to \mathbb{R}^n$ with $\bar{x} (t) = \x = x(t)$ satisfies 
\[
|\bar x(s) - x(s)|<C\delta, \quad s \in [t, T]
\]
for some positive constant $C$. 
\end{prop}
\begin{proof}
Fix an arbitrary $\delta > 0$ and choose $\tilde x$ and $\tilde u\in \mathcal{U}$ from Lemma \ref{lem:relax} such that $|x(s)-\tilde x(s)| <\delta$ for all $s \in [t,T]$ and $\tilde{x}(t) = \x$. 
Consider the set of balls centered at the elements of a $\delta$-net $U_\delta$ with radius $\delta$, that is, $Y:= \{B_\delta(a): a \in U_\delta\}$. 
For simplicity, we
denote $Y=\{B_1, \ldots ,B_\ell\}$, where the center of $B_i$ corresponds to $a_i$. 
Define a new control $\bar{u}$ as
\begin{equation*}
    \bar u(s):=\sum_{i=1}^\ell  a_i  \mathbf{1}_{\tilde B_i}( \tilde u(s)),
\end{equation*}
where $\tilde B_i := B_i \setminus \bigcup_{j=1}^{i-1}B_j$ and 
$\mathbf{1}_B(x)$ is
 an indicator function such that $\mathbf{1}_B(x) = 1$ if $x\in B$ and 0 otherwise. Then, $\bar u(s)$ is measurable and $|\bar u(s) - \tilde u(s)| <\delta$. 
 Let $\bar{x}$ be the state trajectory controlled by $\bar{u}$ starting from $\bar{x}(t) = \x$.
 We then have for any $s \in [t, T]$
\begin{align*}
|\bar x(s) - x(s)|&\leq |\bar x(s)- \tilde x(s) | + |\tilde x(s) - x(s)|\\
&\leq\int_t^s|f (s, \bar x(s), \bar u(s)) - f(s, \tilde x(s), \tilde u(s))| \de s + \delta\\
&\leq \int_t^s C \delta \de s + \delta
\end{align*}
for some positive constant $C$ because $f(s, \cdot, \cdot)$ is Lipschitz continuous under Assumption~\ref{ass3}. 
Thus, the result follows. 
\end{proof}

This proposition provides an important insight into numerical implementation of the representation formulas because the infinite set $U$ can be replaced with a finite set $U_\delta$ when constructing approximate optimal controls that are admissible. 
More specifically, let $(x(\cdot),\beta(\cdot))$ be an optimal solution of \eqref{dual_c} under the state constraint $x(s)\in\bar \Omega$.
Since $\beta(s) \in \Conv(B(x(s),s))$, one cannot directly obtain an admissible control $u \in \mathcal{U}$ such that $\beta(s)=-f(s,x(s),u(s))$.
 Instead, for each $s$, we choose $u(s)$ to be $\argmin_{a \in U_\delta}|\beta(s)+f(s,x(s), a)|$ for a given $\delta$-net $U_\delta$. It is a reasonable approach, as we can always find a feasible control $u \in \mathcal{U}$ such that the corresponding state trajectory approximates the optimal one  by Lemma~\ref{lem:relax}. 
 Motivated by this observation, a computationally tractable scheme for controller synthesis is discussed in Section~\ref{sec:num}. 
 
 Moreover, by Proposition~\ref{prop:u}, we can further approximate an optimal control $u(s)$ by a measurable control $\bar u \in \mathcal{U}$ such that $\bar u(s) \in U_\delta$.
Then, it follows from Proposition~\ref{prop:u} that the performance gap between the optimal control $u$ and the approximate control $\bar u$ is bounded by $C \delta$ for some positive constant $C$.  

\begin{cor}\label{main_cor}
Suppose $H_f^*(t,\cdot)$ is linear and let $(x(\cdot),\beta(\cdot))$ be an optimal solution of \eqref{dual} (or \eqref{dual_c}). Then, for any $\delta >0$, there exists an admissible control $\bar u \in \mathcal{U}$ with $\bar u (s) \in U_\delta$ and
a state trajectory $\bar x : [t,T] \rightarrow \R^n$  satisfying $\dot{\bar{x}} = f(s, \bar{x}(s),\bar{u}(s))$ for $s\in[t,T]$ and $\bar{x}(t) = \x$ 
such that
\begin{equation*}
    \bigg|  \underbrace{\bigg \{\int_t^T H^*(s,x(s),\beta(s))\de s+g(x(T)) \bigg \}}_{= V(\x, t)}- \bigg \{ \underbrace{\int_t^T H^*(s,\bar x(s),\bar \beta(s))\de s-g(\bar x(T))}_{\mathrm{approximation}} \bigg \}\bigg|<C\delta
\end{equation*}
for some positive constant $C$, 
where  $\bar \beta(s)=-f(s,\bar x(s),\bar u(s))$.
\end{cor}

The corollary can be shown using the argument in the proof form Theorem~\ref{main_thm} together with the bound in Proposition~\ref{prop:u}.
Thus, we have omitted the proof.

\subsection{Linearity of $H_f^*(t, b)$ in $b$}\label{sec:lin}
 
 We discuss a class of optimal control problems in which $H^*_f(t,b)$ is linear in $b$ while $f(s,\x,a)$ is nonlinear in $a$. The simplest case is $R(t,a)=0$ so that $L(t,\x,a)$ is independent of the control input. It is straightforward to check that $H^*=0$ in that case. However, it is a too restrictive case. Thus, we aim to find a class of problems with control-dependent cost functions.

Consider a dynamical system of the form
\[
\dot x(t) = (\psi(u(t)),A x(t)) \in \R^n
\]
where $A x(t) \in \R^{n-2}$, $\psi(u(t))\in \R^{2}$, $u(t) \in U \subset \R^{2}$ and $x(t)\in \R^n$. For simplicity, let $U=U_1 \times U_2$ and $U_1=\{q_1,q_2\}$ with $q_i \in \R$ and $U_2$ represents an arbitrary interval in $\R$.
  It is a natural class of control sets as we often consider switched systems where both discrete and continuous control inputs are used. 
  For $a=(a_1,a_2)\in U_1\times U_2$, we further assume 
  \[
  \psi(a)=(\psi_1(a),\psi_2(a))=(f_1(a_1)a_2,f_2(a_1)a_2)
  \]
and $L(s,\x,a)=S(\x)+ca_2$ for some $c \in \R$, where $f_i(\cdot)$'s are functions defined on $U_i$. 
In this setting,  $H_f^*$ is linear in $b$. To see this, let $b_1 := f_1(a_1)a_2$ and $b_2 := f_2(a_1)a_2$. Then, $b_1$ and $b_2$ form a line segment for each $a_1 \in \{q_1,q_2\}$. As a result, we have two line segments represented by $(b_1,b_2)$ that are parametrized over $U=U_1\times U_2$. Since the Lagrangian $L(t,\x, a)=V(\x)+ca_2$ is maximized at the end point of each line segment, $\tilde L(t,\x, b) = S(\x)+ c' \cdot b$ for some $c'\in \R^2$. Therefore, $H_f^*(t, b)$ is linear in $b$.

\section{Numerical Implementation}\label{sec:num}

In this section, we provide a concrete way to numerically solve  optimal control problems using the theoretical results developed in the previous sections. 
Consider the temporal discretization $\{t_k\}_{k=0}^{K}$ with $t=t_0<t_1<...<t_K=T$ and let $\Delta_k:=t_{k+1}-t_{k}$.
As suggested in \cite{lee2021computationally},
the value function $V(\x,0)$ can be approximated by
\begin{equation}\label{disc}
    V(\x, 0) \approx \min_{x[\cdot],\beta[\cdot]} \sum_{k=0}^{K-1} H^*(t_k,x[k],\beta[k])\Delta_k+g(x[K])
\end{equation}
subject to
\begin{align*}
\begin{cases}
    x[k+1]=x[k]-\beta[k]\Delta_k,\\
    \beta [k] \in \Conv(B(x[k],t_k)),\\
    x[k]\in \bar \Omega,\\
    x[0]=\x.
\end{cases}
\end{align*}
This is an optimization problem with respect to $x[\cdot]$ and $\beta [\cdot]$. This problem is a convex optimization problem under Assumption \ref{ass2} and \ref{ass3} when $\Omega$ is convex.\footnote{Although the convexity of the constraint $\beta[k] \in \Conv(B(x[k],t_k)$ is not immediately clear,  it is shown to be a convex set with respect to $(x,\beta)$ (See Lemma 7 in \cite{lee2021computationally})}
 Therefore, the problem can be numerically solved using existing convex optimization algorithms. 
 However, we observe that the resulting control trajectory $u[\cdot]$ often has a chattering issue as demonstrated in \cite[Fig. 3 and 5]{lee2021computationally}. This issue arises from reconstructing the feasible control sequence $u[k] \in B(x[k],t_k)$ using $\beta [k] \in \Conv(B(x[k],t_k))$.
Our goal is to simplify the controller synthesis process without the chattering issue
using the theoretical framework developed in the previous sections.

When solving the optimization problem \eqref{dual} or \eqref{dual_c} in continuous time,  one can  find an admissible control $\tilde u \in \mathcal{U}$ and the corresponding state trajectory $\tilde x$  such that $|x(s)-\tilde x(s)|\approx 0$ using Lemma~\ref{lem:relax}. 
Motivated by this observation, 
 we propose a simple control synthesis scheme in discrete time. 
 Our scheme is summarized as follows: 
 \begin{enumerate}
 \item[(i)] Solve the convex optimization problem \eqref{disc} to obtain $x[k]$ and $\beta[k] \in \Conv(B(x[k],t_k))$;
 
 \item[(ii)] Find $\tilde u[k] \in U$ such that the corresponding state trajectory $\tilde x[k]$ satisfies $|x[k]-\tilde x[k]| < \epsilon$ and $|\beta[k]+f(t_k,\tilde x[k],\tilde u[k])| < \epsilon$, where $\epsilon > 0$ is a user-specified threshold. 
 
 \end{enumerate}
 
 In general,  obtaining $\tilde u[k]$ satisfying the conditions in (ii) is not an easy task. 
As a practical remedy,  we propose an intuitive method that uses a $\delta$-net
 to approximate $\beta[k] \in \Conv(B(x[k],t_k))$ obtained from the optimization problem by $-f(t_k,x[k],\bar u[k])$, where $\bar u[k] \in U_\delta$. Specifically, having found $\beta[k]$ and $x[k]$, we choose $\bar u[k] \in U_\delta$ so that $-f(t_k,x[k],\bar u[k])$ becomes the closest point to $\beta[k]$.
 Using this idea, Step (ii) can be modified as
  \begin{enumerate}
   \item[(ii$'$)] Approximate $\tilde u[k]$ as
 \[
 \tilde u [k] \approx \bar u[k] \in \argmin_{a \in U_\delta}\left |\beta[k]+f(t_k,x[k],a) \right |.
 \] 
  \end{enumerate}

When using the method proposed in \cite{lee2021computationally},
approximated control trajectories fluctuate frequently even though 
the trajectory of $\beta$ has no chattering issue. 
This is because the algorithm in \cite{lee2021computationally}
uses interpolation at every time step. 
Instead, our method   chooses the nearest control input to obtain an admissible control that is less likely to display fluctuations as long as $\beta(t) \in \Conv(B(x(t), t))$ is not oscillating wildly.

\begin{figure}[t]
     \centering
     \begin{subfigure}[b]{0.4\linewidth}
         \centering
         \includegraphics[width=\linewidth]{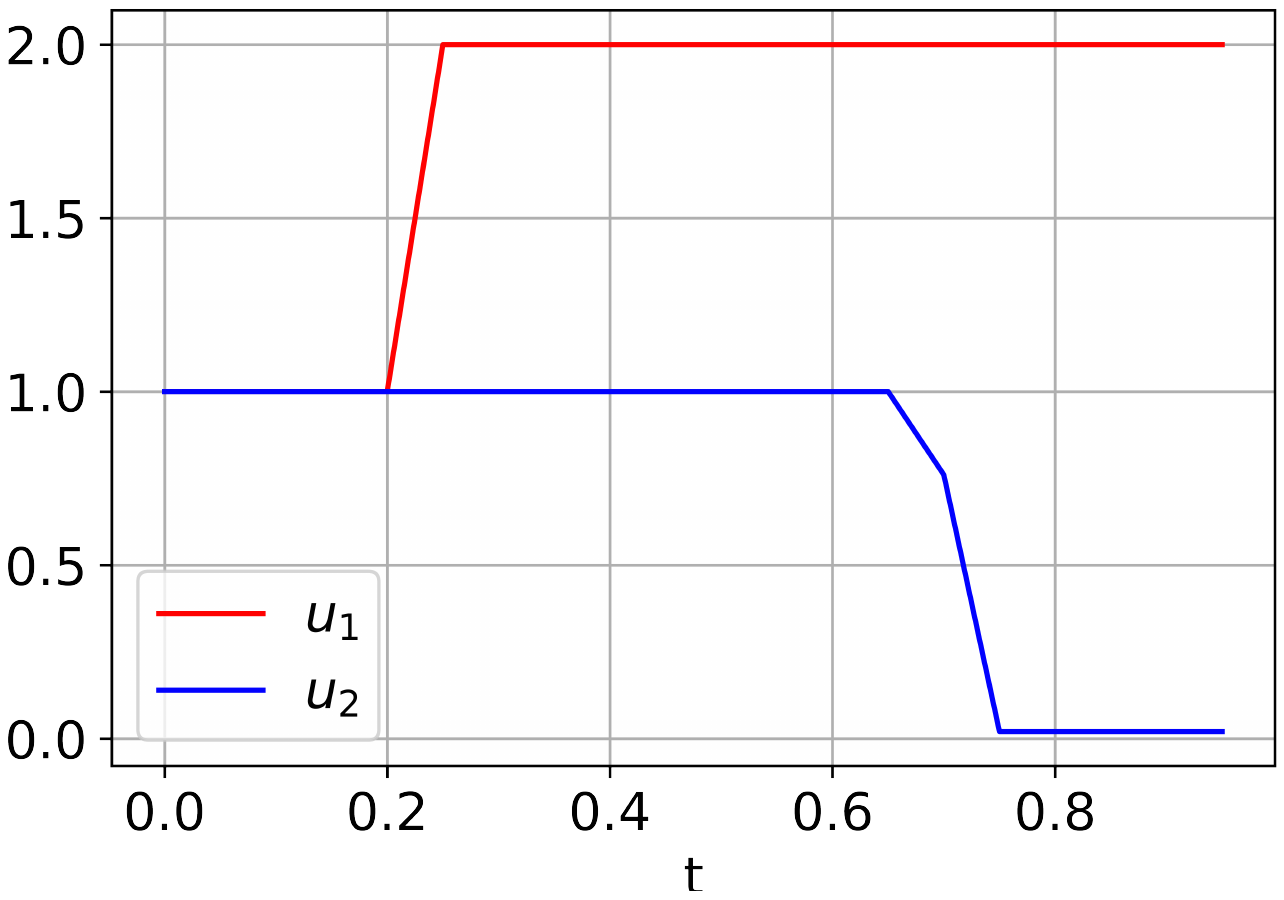}
         \caption{Lagrangian (ours)}
         \label{fig:state_u_1}
     \end{subfigure}%
     \hspace{0.5in}
     \begin{subfigure}[b]{0.4\linewidth}
         \centering
         \includegraphics[width=\linewidth]{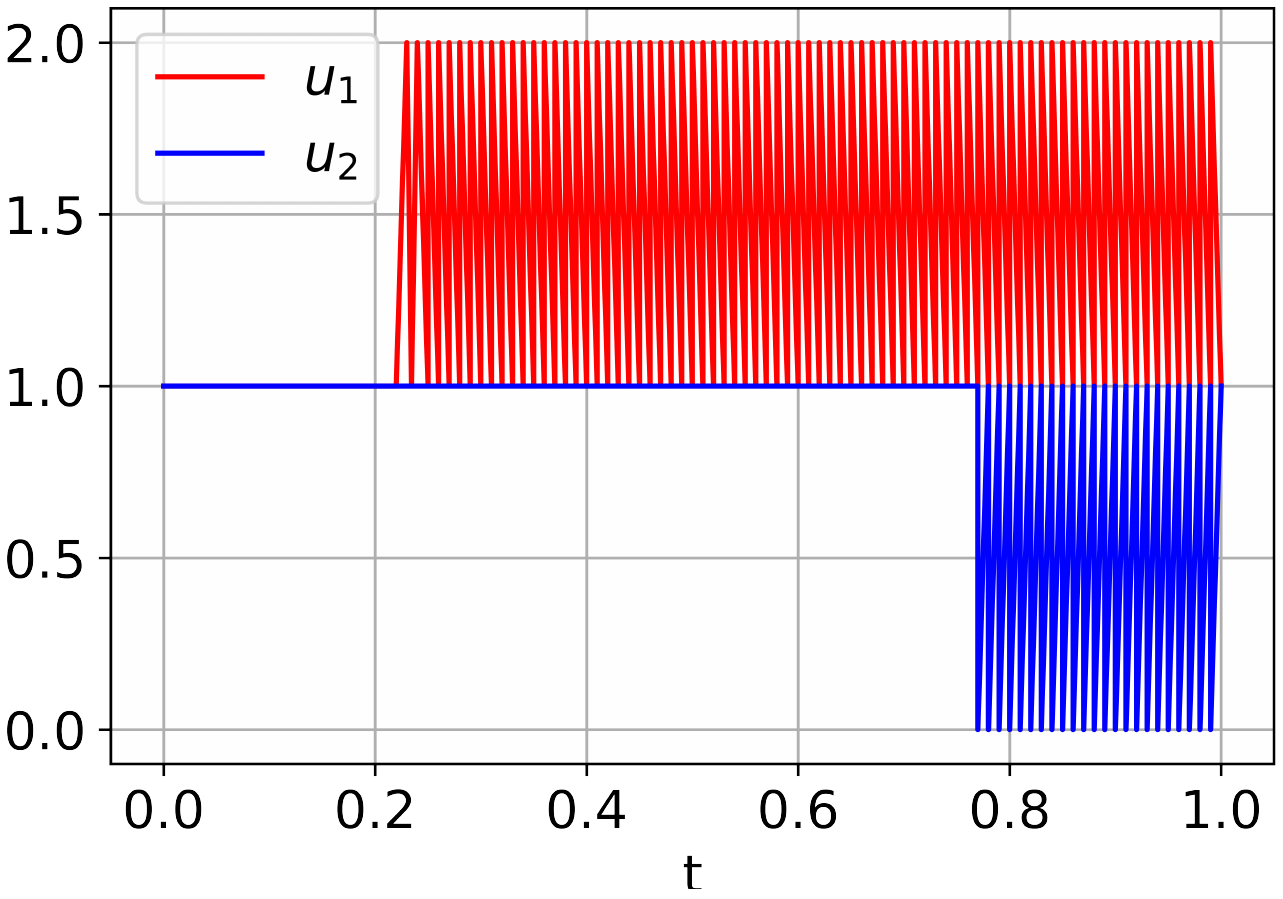}
         \caption{Method in  \cite{lee2021computationally}}
         \label{fig:state_u_2}
     \end{subfigure}
    \caption{Control trajectories obtained by (a) our method and (b)  the method in \cite{lee2021computationally}.}
    \label{fig:control}
\end{figure}

 \begin{figure}[t]
\centering
     \begin{subfigure}[b]{0.4\linewidth}
         \centering
         \includegraphics[height=1.75in]{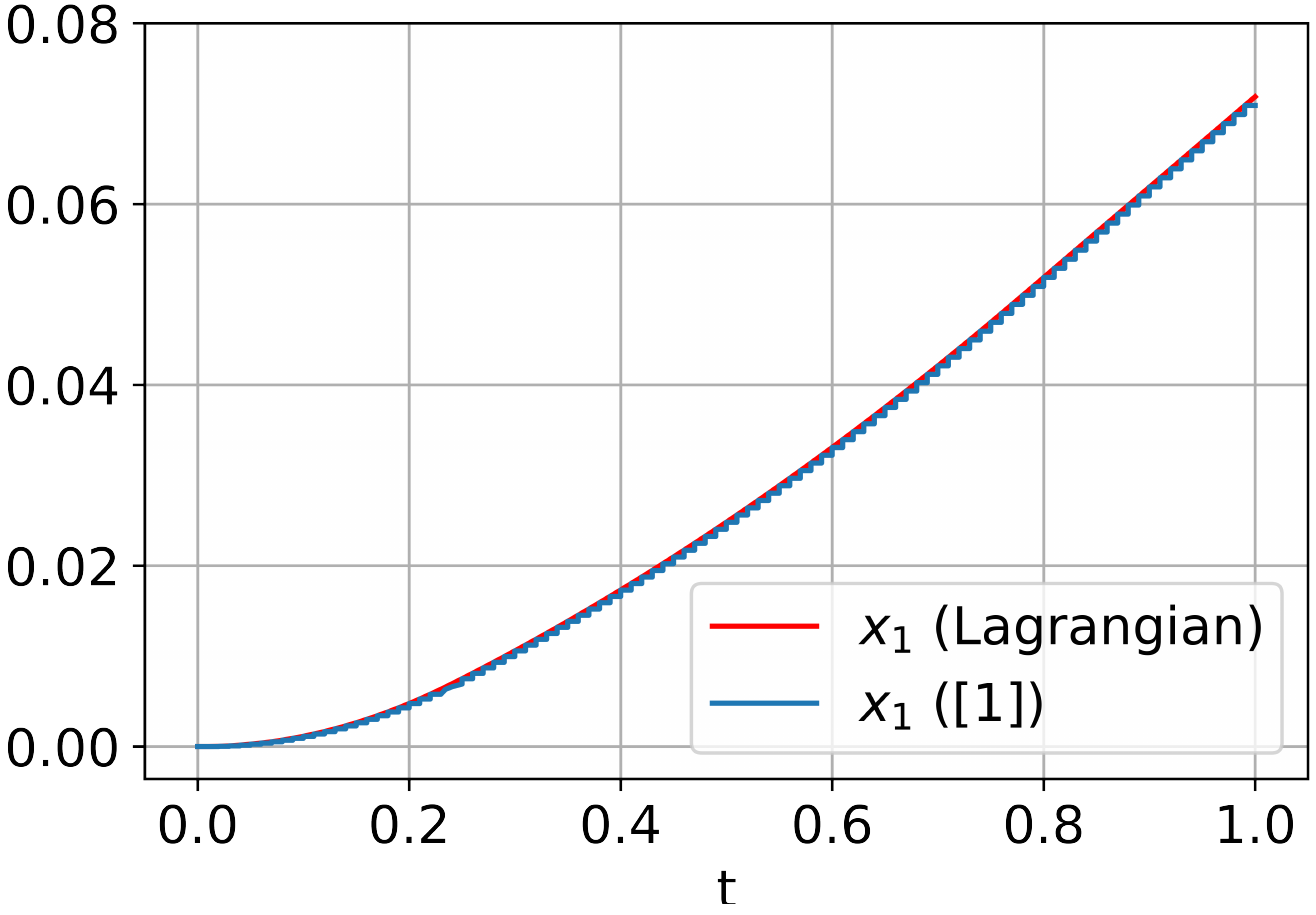}
         \caption{$x_1$}
     \end{subfigure}%
          \hspace{0.3in}
     \begin{subfigure}[b]{0.4\linewidth}
         \centering
         \includegraphics[height=1.75in]{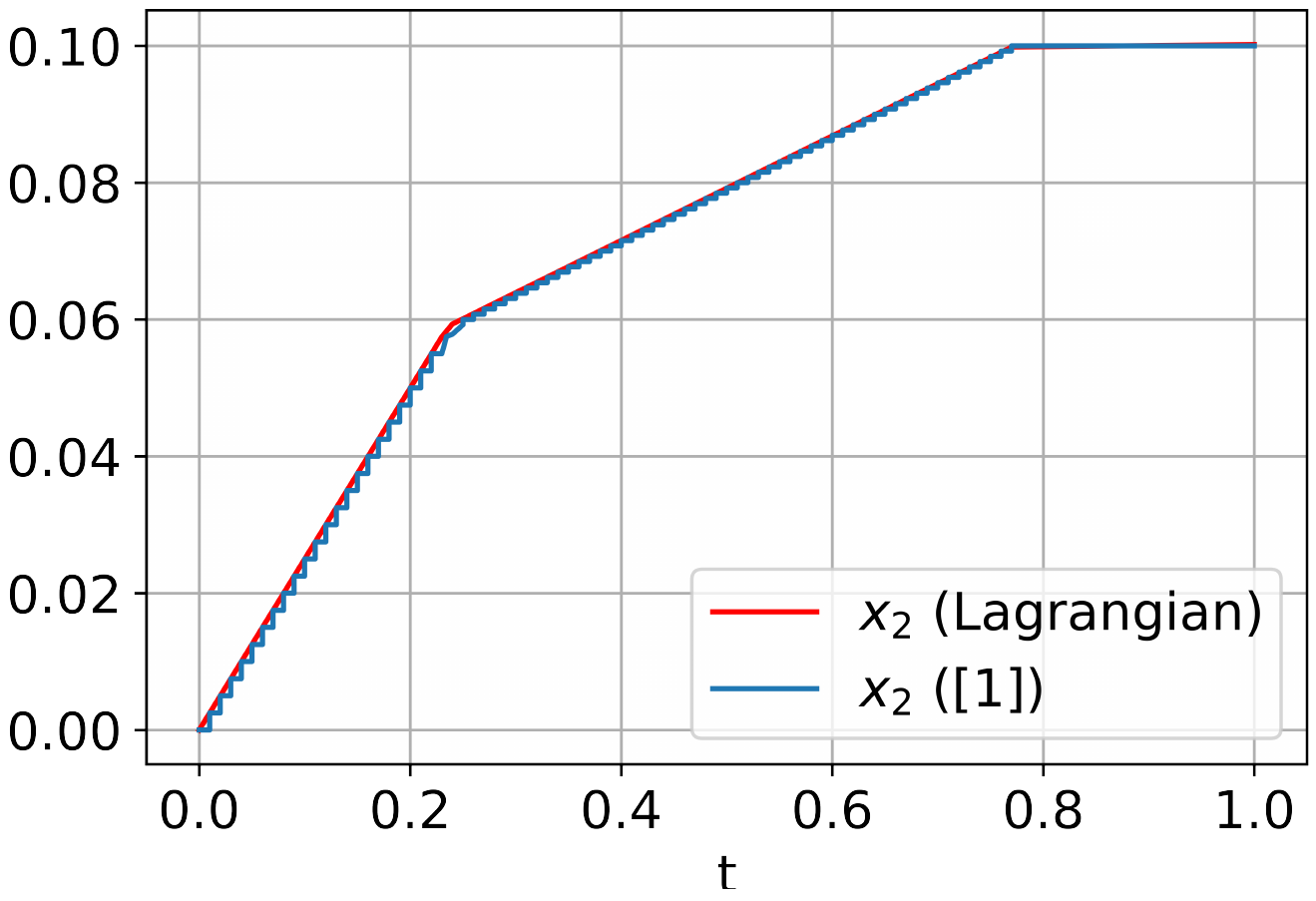}%
         \caption{$x_2$}
     \end{subfigure}\\
     \vspace{0.2in}
     \begin{subfigure}[b]{0.4\linewidth}
         \centering
         \includegraphics[height=1.75in]{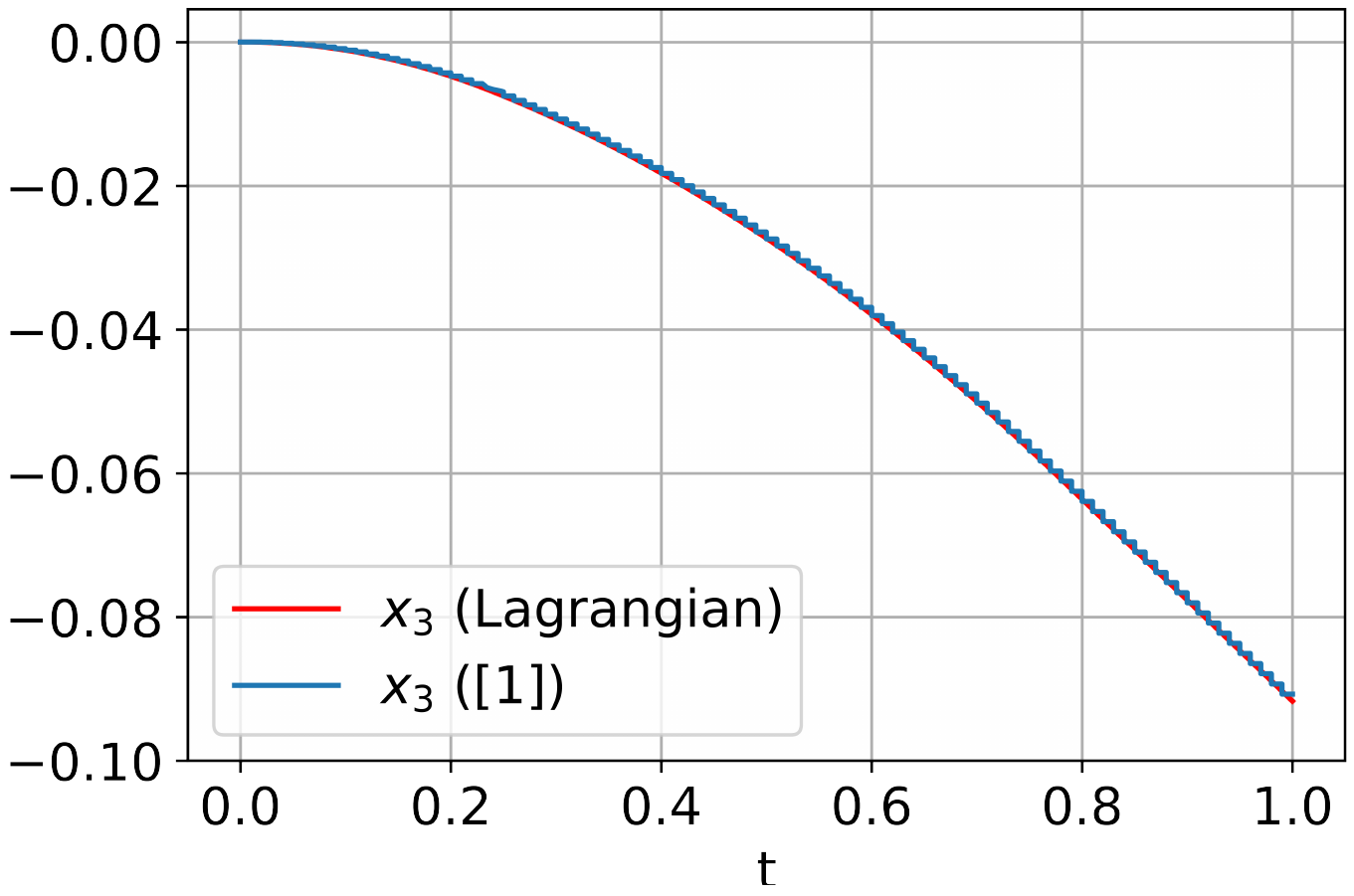}
         \caption{$x_3$}
     \end{subfigure}%
          \hspace{0.3in}
     \begin{subfigure}[b]{0.4\linewidth}
         \centering
         \includegraphics[height=1.75in]{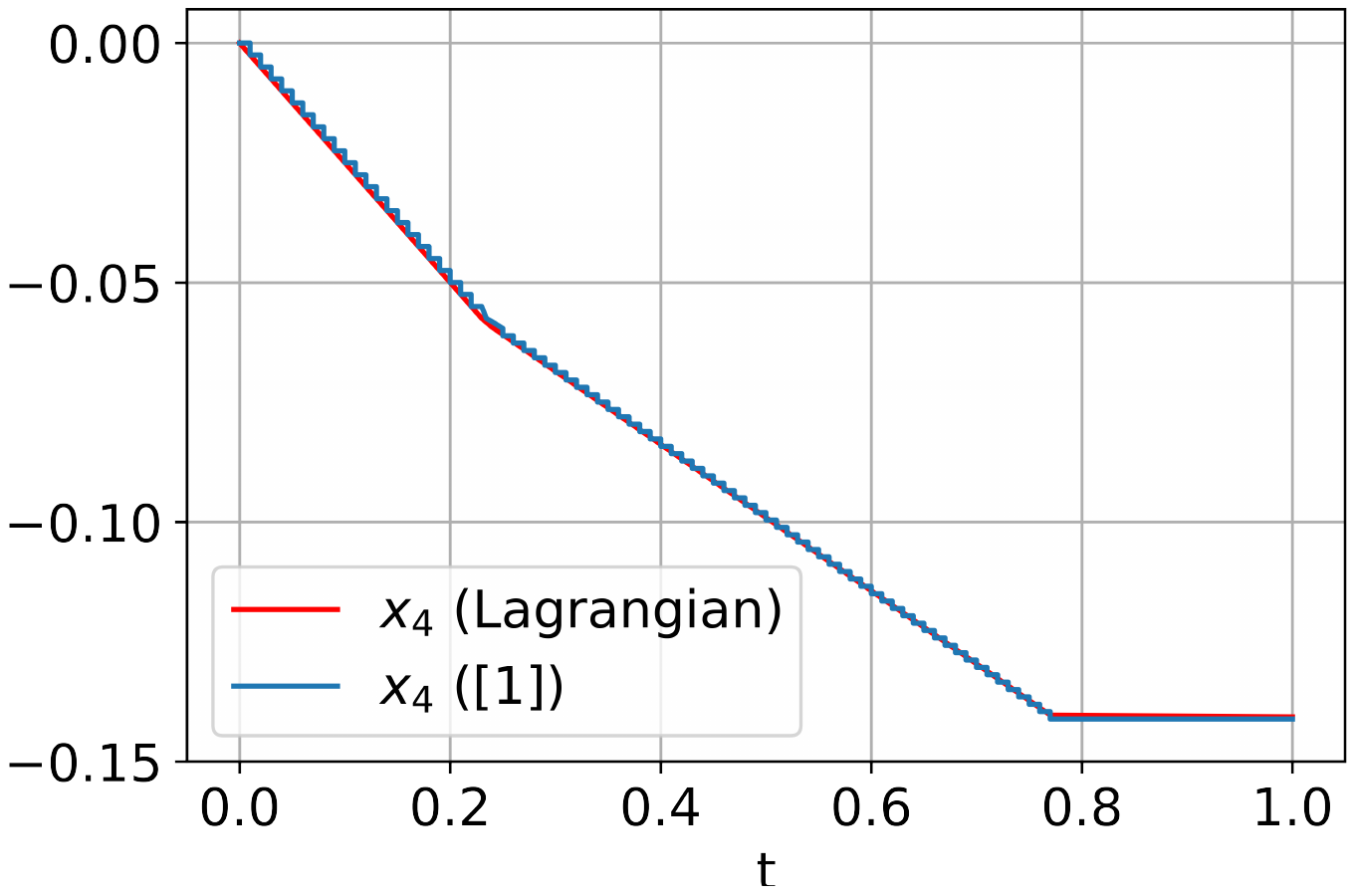}
         \caption{$x_4$}
     \end{subfigure}
        \caption{Controlled state trajectories obtained by our scheme and the method in \cite{lee2021computationally}.}
       \label{fig:state}
\end{figure}

To demonstrate the performance of our controller synthesis method, we consider an example with system dynamics
\begin{align*}
f(s,x(t),u(t))=\left(x_2(t),\frac{1}{1+3u_1(t)^2}u_2(t),x_4(t),-\frac{u_1(t)}{1+3u_1^2(t)}u_2(t)\right),
\end{align*}
where $x(t)=(x_1(t),x_2(t),x_3(t),x_4(t)) \in \R^4$ and $u(t)=(u_1(t),u_2(t))\in \{1,2\} \times [0,1]$.
This system models the dynamics of a gear system~\cite{lee2021computationally}. 
The cost function is selected as
\begin{equation*}
\int_0^1 u_2(s) \de s+1000x_3(1),
\end{equation*}
and the following state constraint is imposed:
\[
|x_2 | \leq 0.1.
\]
Although the vector field is nonlinear in $u(t)$, the convex conjugate of the Hamiltonian is linear in $b$ since $H^* (t, \x, b) = 5b_2 + 9b_4$.

A $\delta$-net of $U$ is constructed as the set of 100 points that are obtained uniformly discretizing $U$. The initial state $x(0) = \bm{x}$ is chosen as the origin. All experiments were performed using a PC  with 3.3GHz 14-core i9 CPU and 64 GB RAM. CVXPY on Python 3.8.8 was used to numerically solve
the convex optimization problems.

   We compare the results of our scheme and the method proposed in \cite{lee2021computationally}.
   Figure~\ref{fig:control} shows the control trajectories obtained by the two methods with $\Delta_k = 0.01$. 
   It is remarkable that our scheme significantly reduces the oscillations in the trajectories generated by the method in \cite{lee2021computationally}.
   Despite the difference in the control trajectories, 
   the corresponding trajectories are  similar to each other      as shown in Figure~\ref{fig:state}.
   Note also that the state constraint is satisfied for all time. 
       The resulting total costs are compared in Table~\ref{tab:perf}, showing that the controller constructed by our scheme slightly outperforms that obtained by the method in \cite{lee2021computationally}.
  In terms of computation time, the two methods are comparable, taking less than 1 second whenever $\Delta_k \geq 0.01$.

\begin{table*} 
	\caption{Quantitative comparisons of the two methods in terms of the total cost and the computation time (in seconds).}
	\begin{center}
			\begin{tabular}{| l | >{\centering}p{1.5cm} |
			| >{\centering}p{3cm} |>{\centering\arraybackslash}p{3cm} |}
				\hline
  & Stepsize $\Delta_k$ & {\textbf{Lagrangian (ours)}}  & \textbf{Method in \cite{lee2021computationally}} \\\hline\hline
				  \multirow{3}{*}{Total cost}  & $0.05$      & $-88.5358$ & $-81.5702$\\ 			  
& $0.02$ &  $-90.7164$ & $-87.8400$	 \\	
& $0.01$ & $-90.7164$ & $-89.9517$	 \\			  \hline
			 \multirow{3}{*}{Comp. time (sec)}   &$0.05$  & $0.1962$ & $0.1952$ \\		
&$0.02$  & $0.4509$ & $0.4326$  \\
&$0.01$  & $0.8608$ & $0.8500$ \\		 \hline
\end{tabular}
		\label{tab:perf}
	\end{center}
\end{table*}

\section{Conclusions}

We have shown that the generalized Lax formula is essentially equivalent to the Lagrangian formula from the theory of HJ equations
 when taking the convex conjugate of the Hamiltonian as the Lagrangian. 
Our analysis provides a new way to understand and implement the representation formula. 
Specifically, 
our theoretical results allow us to propose a rigorous process to construct an admissible control with a performance guarantee as well as a computationally tractable controller synthesis scheme.
The performance of our scheme has been demonstrated through a state-constrained nonconvex optimal control example, showing that it resolves the control chattering issue observed in the previous work. 

This work is a stepping stone to more extensive  results. 
Specifically, it is worth relaxing the structural assumptions used in constructing admissible controls to solve a larger class of problems. 
Moreover, numerical schemes for controller synthesis may be refined with a rigorous convergence analysis bridging the gap between the continuous-time controller and its discrete-time approximation.

\section*{Acknowledgement}
The authors are grateful to Prof. Hung Vinh Tran at Univ. of Wisconsin-Madison and Dr. Donggun Lee at UC Berkeley for insightful discussions.

%

\bibliographystyle{IEEEtran}

\bibliography{ref}

\end{document}